\documentclass[11 pt]{article}   	

\usepackage{amssymb,amsmath,amsthm,mathtools}
\usepackage{hyperref}
\usepackage{xcolor,url,graphicx,epstopdf}
\usepackage{cite}
\usepackage{authblk}

\hypersetup{
    colorlinks,
    linkcolor={red!80!black},
    citecolor={green!80!black},
    urlcolor={blue!80!black}
}

\usepackage[letterpaper,bindingoffset=0.2in,%
            left=1in,right=1in,top=1in,bottom=1in,%
            footskip=.25in]{geometry}
            
\newtheorem{theorem}{Theorem}

\newtheorem{lemma}{Lemma}
\newtheorem{definition}{Definition}
\newtheorem{remark}{Remark}

\def\e{\epsilon}
\def\grad{\nabla}
\def\diamX{\text{diam}(\mathcal{X})}

\def\Fcal{\mathcal{F}}
\def\Xcal{\mathcal{X}}

\def\Ebb{\mathbb{E}}
\def\Pbb{\mathbb{P}}

\def\ghat{\hat{g}}
\def\rhat{\hat{r}}

\title{On the modes of convergence of Stochastic Optimistic Mirror Descent (OMD) for saddle point problems}
\author[1]{Yanting Ma}
\author[2]{Shuchin Aeron}
\author[1]{Hassan Mansour}
\affil[1]{Mitsubishi Electric Research Laboratories (MERL)}
\affil[2]{Department of Electrical and Computer Engineering, Tufts University}
\date{}

\begin{document}
\maketitle

\begin{abstract}
In this article, we study the convergence of Mirror Descent (MD) and Optimistic Mirror Descent (OMD) for saddle point problems satisfying the notion of coherence as proposed in \cite{Mertikopoulos2019optimistic}. We prove convergence of OMD with exact gradients for coherent saddle point problems, and show that monotone convergence only occurs after some sufficiently large number of iterations. This is in contrast to the claim in~\cite{Mertikopoulos2019optimistic} of monotone convergence of OMD with exact gradients for coherent saddle point problems. Besides highlighting this important subtlety, we note that the almost sure convergence guarantees of MD and OMD with stochastic gradients for strictly coherent saddle point problems that are claimed in \cite[Theorems 3.1 and 4.3]{Mertikopoulos2019optimistic}, respectively, are not fully justified by their proof. As such, we fill out the missing details in the proof and as a result have only been able to prove convergence with high probability. 

We would like to note that our analysis relies heavily on the core ideas and proof techniques introduced in \cite{Zhou2017stochasticMD, Mertikopoulos2019optimistic} and only aim to re-state and correct the results in light of what we were able to prove rigorously while filling in the much needed missing details in their proofs.


\end{abstract}

\section{Introduction}
We analyze some recent results on the use of Mirror Descent (MD) and Optimistic Mirror descent (OMD), \cite{Mertikopoulos2019optimistic} that have recently been studied extensively for alleviating convergence issues in training of adversarial generative networks \cite{Daskalakis2018training,Mertikopoulos2019optimistic}. In particular these papers consider the following general saddle-point (SP) problem.
\begin{align} \label{eq:OPT}
   \min_{x_1\in\Xcal_1}\max_{x_2\in\Xcal_2} f(x_1,x_2),
\end{align}
where $\Xcal_i$ are compact convex subset of a finite-dimensional normed space $\mathcal{V}_i\in\mathbb{R}^{d_i}$, $i=1,2$, and $f:\Xcal_1\times\Xcal_2\to\mathbb{R}$ is continuously differentiable. Let $\Xcal:=\Xcal_1\times\Xcal_2$ and let $\mathcal{V}^*$ be the dual of $\mathcal{V}:=\mathcal{V}_1\times\mathcal{V}_2$. Define $g:\Xcal\to\mathcal{V}^*$ as
\begin{equation*}
    g(x):= (\grad_{x_1} f(x_1,x_2), - \grad_{x_2} f(x_1,x_2)).
\end{equation*}
It is well-known that if $x^*$ is a solution to \eqref{eq:OPT}, then it satisfies the Stampacchia Variational Inequality (SVI) \cite{Rosa07}, i.e.
\begin{equation*}
    \langle g(x^*), x - x^* \rangle \geq 0, \forall x\in\Xcal.
\end{equation*}
When the function $f(x,y)$ is convex-concave, it is also well-known that the point $x^*\in\Xcal$ also satisfies the Minty Variational Inequality (MVI) \cite{Rosa07}, i.e. 
\begin{equation*}
    \langle g(x), x - x^* \rangle \geq 0, \forall x\in\Xcal.
\end{equation*}
For convex-concave problems, it can be shown that these three conditions, namely, $x^*$ is the optimal point of \eqref{eq:OPT}, $x^*$ satisfies SVI, and $x^*$ satisfies MVI, are all equivalent. The interplay of these two VIs characterizing the optimal set of solutions has been investigated extensively. We focus here on the notion of coherence proposed in \cite{Zhou2017stochasticMD,Mertikopoulos2019optimistic}, that are assumed to be satisfied by problem \eqref{eq:OPT} and that are in some sense the weakest set of possible conditions considered for global optimality going beyond convexity, pseudo-montonocity, and quasi-convexity \cite{Zhou2017stochasticMD}.

In order to describe MD and OMD, one needs to define the notion of Bregman Divergence (BD) with respect to a differentiable and $K$-strongly convex function $h$ whose domain includes the set $\mathcal{X}$. There are several equivalent definitions of $K$-strong convexity, here we provide the one that will directly be used in the proof later:
\begin{equation}
    \langle \grad h(x) - \grad h(x'), x - x' \rangle \geq K\|x - x'\|^2,\quad \forall x,x'\in \Xcal.
\label{eq:def_K}
\end{equation}
We further assume that $\grad h$ is $L_h$-Lipschitz, which is needed in the proof.
The BD is defined as,
\begin{align}
    D(x,y) = h(y) - h(x) - \langle \grad h(x), y - x\rangle.
\end{align}
Bregman divergence enjoys a number of properties that are critical to the success of MD and OMD and we refer the reader to the Apendices in \cite{Mertikopoulos2019optimistic} where they are proposed and derived. Given a vector  $y \in \mathcal{V}^*$, and a vector $x \in \mathcal{V}$, define the following Bregman projection operator $P_x(y)$ via, 
\begin{align}
    P_x(y) = \arg \min_{x' \in \mathcal{X}} \{ \langle y, x - x' \rangle + D(x',x)\} .
    \label{eq:def_Proj}
\end{align}
In the following we will assume that an $h$ is given and fixed throughout.

\subsection{Variational Inequalities and Coherence}
The following definition of coherence is provided in \cite[Definition 2.1]{Mertikopoulos2019optimistic}, where we explicitly define what it means by $x$ being sufficiently close to $x^*$ in Condition 3; this definition will be used in the proofs presented in Section \ref{sec:appendix}.

\begin{definition} 
\label{def:coherence}
We say that (SP) is \textbf{coherent} if
\begin{enumerate}
    \item Every solution of (SVI) also solves (SP).
    \item There exists a solution $p$ of (SP) that satisfies (MVI).
    \item Every solution $x^*$ of (SP) satisfies (MVI) locally. Specifically, for some fixed $\e_0>0$, $\langle g(x), x-x^* \rangle\geq 0$ for all $x\in\Xcal$ such that $D(x^*,x)\leq \e_0$.
\end{enumerate}
If, moreover, (MVI) holds as a strict inequality in Condition 2 when $x$ is not a solution of (SP), then (SP) will be called \textbf{strictly coherent}; by contrast, if (MVI) holds as an equality in Condition 2 for all $x\in\Xcal$, we will say that (SP) is \textbf{null-coherent}.
\end{definition}



\subsection{Mirror Descent Algorithms}
Under stochastic gradients, the MD and OMD are defined below along with the almost standard assumptions on the expected values and the variance of the gradient estimates. For all the probabilistic statements in this article, we consider the probability space $(\Omega,\Fcal,\Pbb)$ and let $\Ebb$ denote the expectation with respect to $\Pbb$.\\

\noindent \textbf{Mirror Descent (MD)}: The vanilla mirror descent (MD) algorithm is defined as
\begin{equation}
X_{n+1} = P_{X_n}(-\gamma_n \ghat_n),
\label{eq:SMD_algo}
\end{equation}
where $G\in [0,\infty)$ is some constant and $\ghat_n$ satisfies 
\begin{equation}
\Ebb[\ghat_n | \Fcal_n]=g(X_n) \quad \text{and} \quad \Ebb[\|\ghat_n\|_*^2 | \Fcal_n]\leq G^2
\label{eq:SMD_grad_cond}
\end{equation}
with $\Fcal_n:=\sigma(X_1,\ldots,X_n)$, the $\sigma$-algebra generated by $X_1,\ldots,X_n$. 
Note that with the assumption in \eqref{eq:SMD_grad_cond}, we have that there exists some constant $\tau$ such that
\begin{equation}
\Ebb[ \| \ghat_n - g(X_n)\|_*^2 | \Fcal_n ] \leq 2\Ebb[ \| \ghat_n\|_*^2 | \Fcal_n ] +  2\| g(X_n)\|_*^2 \leq \tau^2,
\label{eq:def_tau}
\end{equation}
since a continuous function ($g$) is bounded on a compact set ($\Xcal$).\\

\noindent \textbf{Optimistic Mirror Descent (OMD)}: The optimistic mirror descent (OMD) algorithm is defined as
\begin{equation}
Y_n = P_{X_n}(-\gamma_n \ghat_n),\quad X_{n+1} = P_{X_n}(-\gamma_n \rhat_n),
\label{eq:SOMD_algo}
\end{equation}
with the assumption that
\begin{equation}
\begin{split}
&\Ebb[\ghat_n | \Fcal_{n,n-1}] = g(X_n), \quad \Ebb[\rhat_n | \Fcal_{n,n}] = g(Y_n),\\
&\Ebb[\|\ghat_n\|_*^2 | \Fcal_{n,n-1}] \leq G^2, \quad \Ebb[\|\rhat_n\|_*^2 | \Fcal_{n,n}] \leq G^2,
\end{split}
\label{eq:grad_assumption}
\end{equation}
where $\Fcal_{n_1,n_2}=\sigma(X_1,\ldots,X_{n_1},Y_1,\ldots,Y_{n_2})$. Similarly, there exist some finite constant $\sigma$ such that
\begin{equation}
\Ebb[ \| \ghat_n - g(X_n)\|_*^2 | \Fcal_{n,n-1}]  \leq \sigma^2 \quad \text{and} \quad \Ebb[ \| \rhat_n - g(Y_n)\|_*^2 | \Fcal_{n,n}]  \leq \sigma^2.
\label{eq:def_sigma}
\end{equation}

\section{Main Results}

We are now ready to \emph{re-state} the main results from \cite{Mertikopoulos2019optimistic} with several important corrections.

First, for \textbf{coherent problems} and OMD with exact (non-stochastic) gradients, it is claimed in \cite{Mertikopoulos2019optimistic} that $\{D(x^*,X_n)\}_n$ is monotone decreasing for some $x^*\in\Xcal^*$. However, as we fill out the missing details in the proof, we believe that monotone decreasing is guaranteed only after some sufficiently large number of iterations; our modified statement is given in Theorem \ref{thm:OMD}.
\begin{theorem}[{Modified from \cite[Theorem 4.1]{Mertikopoulos2019optimistic}}]
\label{thm:OMD}
Suppose that (SP) is coherent and $g$ is $L_g$-Lipschitz. If the algorithm \eqref{eq:SOMD_algo} is run with exact (non-stochastic) gradient and step-size sequence $\{\gamma_n\}_n$ that satisfies
\begin{equation*}
0 < \lim_{n\to\infty} \gamma_n \leq \sup_{n} \gamma_n < K/L_g,
\end{equation*}
where $K$ is defined in \eqref{eq:def_K}, then $\lim_{n\to\infty} X_n = x^*\in\Xcal^*$. Moreover, there exists some sufficiently large $n_0$ such that $D(x^*,X_n)$ decreases monotonically in $n$ for all $n\geq n_0$.
\end{theorem}

Next, for \textbf{strictly coherent problems}, almost sure convergence is claimed in \cite{Mertikopoulos2019optimistic} for both MD and OMD with stochastic gradients. However, as we fill out the missing details in their proofs, we believe that only convergence with high probability can be guaranteed; our modified theorem statements for MD and OMD are given in Theorem~\ref{thm:SMD} (a) and Theorem~\ref{thm:SOMD}, respectively. Moreover, for \textbf{null-coherent problems}, it is claimed in \cite{Mertikopoulos2019optimistic} that the sequence $\{\Ebb[D(x^*,X_n)]\}_n$ is non-decreasing for all $x^*\in \Xcal^*$, whereas we believe that it is true only for the $x^*$'s that satisfy Condition 2 in Definition \ref{def:coherence}; our modified statement is given in Theorem~\ref{thm:SMD} (b).

\begin{theorem}[{Modified from \cite[Theorem 3.1]{Mertikopoulos2019optimistic}}]
\label{thm:SMD}

Suppose that MD \eqref{eq:SMD_algo} is run with a gradient oracle satisfying \eqref{eq:SMD_grad_cond}. 
\begin{enumerate}
\item[(a)] For strictly coherent problems, for any $\e>0$, if the step-size $\gamma_n$ satisfies $\sum_{n=1}^\infty \gamma_n=\infty$ and
\begin{equation}
\sum_{n=1}^\infty \gamma_n^2 \leq \min\left\{ \frac{\delta\e^2}{2\diamX^2\tau^2}, \frac{K\delta\e}{2\tau^2} \right\},
\label{eq:gamma_cond_MD}
\end{equation}
where $\diamX:=\sup_{x,y\in\Xcal} \|x-y\|$, $\delta\in(0,1)$, $K$ is defined in \eqref{eq:def_K}, and $\tau^2$ is defined in \eqref{eq:def_tau}, then
\begin{equation*}
\Pbb\Big( \exists n_0\in\mathbb{N},\exists x^*\in\Xcal^*, \text{ s.t. }  D(x^*, X_n)\leq \e,\forall n\geq n_0 \Big) \geq 1-\delta.
\end{equation*}

\item[(b)] For null-coherent problems, the sequence $\{\Ebb[D(p,X_n)]\}_n$ is non-decreasing for saddle points $p$ that satisfy Condition 2 (global MVI) in Definition~\ref{def:coherence}.
\end{enumerate}
\end{theorem}

\begin{remark}
Theorem \ref{thm:SMD} (b) implies that for null-coherent problems with a unique saddle point $x^*$ (thus necessarily satisfies global (MVI) by definition of coherence), such as the two-player zero-sum game example provided in \cite[Proposition C.3]{Mertikopoulos2019optimistic}, the sequence $\{\Ebb[D(x^*,X_n)]\}_n$ is non-decreasing.
\end{remark}

\begin{theorem}[{Modified from \cite[Theorem 4.3]{Mertikopoulos2019optimistic}}]
\label{thm:SOMD}
Suppose that (SP) is strictly coherent and stochastic OMD \eqref{eq:SOMD_algo} is run with a gradient oracle satisfying \eqref{eq:grad_assumption}. For any $\e>0$, if the step-size $\gamma_n$ satisfies $\sum_{n=1}^{\infty}\gamma_n=\infty$ and 
\begin{equation}
\sum_{n=1}^\infty \gamma_n^2 \leq \min\left\{ \frac{\delta\e^2}{3 \diamX^2\sigma^2},  \frac{K\delta\e}{3 \sigma^2}  \right\},
\label{eq:gamma_cond}
\end{equation}
where $\delta\in(0,1)$, $K$ is defined in \eqref{eq:def_K}, and $\sigma^2$ is defined in \eqref{eq:def_sigma}, then 
\begin{equation}
\Pbb\Big(  \exists n_0\in\mathbb{N}, \exists x^*\in \Xcal^*, \text{ s.t. } D(x^*,X_n) \leq \e, \forall n\geq n_0  \Big) \geq 1- \delta.
\label{eq:SOMDresult}
\end{equation}
\end{theorem}

\begin{remark}
The conditions on the step-size sequence $\{\gamma_n\}_n$ given in \eqref{eq:gamma_cond_MD} and \eqref{eq:gamma_cond} suggest that there is a trade-off between the evolution speed of the algorithm (how large $\gamma_n$ can be), the accuracy of the solution (how small $\e$ can be), and the probability of convergence (how small $\delta$ can be).
\end{remark}

\section{Conclusions and Future Work}
In an attempt towards understanding the recent body of work on MD/OMD dynamics for saddle point problems, in this article we have provided more rigorous and corrected statement of the claims in \cite{Mertikopoulos2019optimistic}. As part of future work we aim to shed light on the rates of convergence of MD and OMD under coherency assumptions. In this context we aim to build upon the analysis conducted in \cite{Mertikopoulos2018stochastic}.

\bibliographystyle{ieeetr}
\bibliography{cites}

\newpage
\section{Appendix}
\label{sec:appendix}

\subsection{Proof of Theorem \ref{thm:OMD}}
First, we restate \cite[Lemma D.1]{Mertikopoulos2019optimistic}.

\begin{lemma}
Suppose that (SP) is coherent and g is $L$-Lipschitz. For any saddle point $x^*\in \Xcal^*$, the iterates of (OMD) with exact gradient satisfy 
\begin{equation}
D(x^*,X_{n+1}) \leq D(x^*,X_n)  - \frac{1}{2}\left(K - \frac{\gamma_n^2 L^2}{K} \right) \|Y_n - X_n\|^2 - \gamma_n \langle g(Y_n), Y_n - x^*\rangle.
\label{eq:lemD1_1}
\end{equation}
If moreover, $p\in\Xcal^*$ is (one of) the special saddle points that satisfy (MVI) globally, then 
\begin{equation}
D(p,X_{n+1}) \leq D(p,X_n) - \frac{1}{2}\left(K - \frac{\gamma_n^2 L^2}{K} \right) \|Y_n - X_n\|^2.
\label{eq:lemD1_2}
\end{equation}
\label{lem:D1}
\end{lemma}

\begin{proof}
Note that no additional proof is needed, since \eqref{eq:lemD1_1} is directly obtained from the first inequality in \cite[(D.2)]{Mertikopoulos2019optimistic} and \eqref{eq:lemD1_1} is the original statement of \cite[Lemma D.1]{Mertikopoulos2019optimistic}.
\end{proof}

Next, we add a result that is similar to \cite[Proposition B.4(a)]{Mertikopoulos2019optimistic}.
\begin{lemma}
Let $h$ be a $K$-strongly convex distance-generating function on $\Xcal$ and further assume that $\grad h$ is $L_h$-Lipschitz. Then, for any $y\in\mathcal{Y}$, we have
\begin{equation*}
\|P_{x_1}(y) - P_{x_2}(y)\| \leq \frac{L_h}{K}\|x_1 - x_2\|, \quad \forall x_1,x_2\in\text{dom } \partial h.
\end{equation*}
\label{lem:propB4ext}
\end{lemma}
\begin{proof}
Let $z_1=P_{x_1}(y)$ and $z_2=P_{x_2}(y)$. By \cite[Lemma B.1(b)(c)]{Mertikopoulos2019optimistic}, we have
\begin{align*}
&\langle \grad h(z_1) - y - \grad h(x_1), z_1 - p \rangle \leq 0, \quad \forall p\in \Xcal\\
&\langle \grad h(z_2) - y - \grad h(x_2), z_2 - p' \rangle \leq 0, \quad \forall p'\in \Xcal.
\end{align*}
Letting $p=z_2$ and $p'=z_1$ and adding the two inequalities:
\begin{equation*}
\langle \grad h(z_1) - \grad h(z_2), z_1 - z_2\rangle \leq \langle \grad h(x_1) - \grad h(x_2), z_1 - z_2 \rangle.
\end{equation*}
Note that the LHS of the above is low bounded by
\begin{equation*}
\langle \grad h(z_1) - \grad h(z_2), z_1 - z_2\rangle \geq K \|z_1 - z_2\|^2
\end{equation*}
by strong convexity of $h$. The RHS of the above is upper bounded by
\begin{equation*}
\langle \grad h(x_1) - \grad h(x_2), z_1 - z_2 \rangle \leq L_h\|x_1-x_2\|\|z_1 - z_2\|
\end{equation*}
by Cauchy-Schwarz and Lipschitz property of $\grad h$. Combining the two, we have
\begin{equation*}
\|z_1 - z_2\| \leq \frac{L_h}{K}\|x_1 - x_2\|.
\end{equation*}
\end{proof}

We are now ready to prove Theorem \ref{thm:OMD}.
\begin{proof}
Let $p$ be a saddle point of (SP) and satisfies (MVI) globally. Such $p$ exists by definition of coherence. By \eqref{eq:lemD1_2} in Lemma \ref{lem:D1}, we have
\begin{equation*}
D(p,X_{n+1}) \leq D(p,X_n) - \frac{1}{2}\left(K - \frac{\gamma_n^2 L_g^2}{K} \right) \|Y_n - X_n\|^2.
\end{equation*}
Since $\sup_n \gamma_n < K/L_g$, there exists $\alpha\in(0,1)$ such that $\gamma_n < \alpha K/L_g$ for all $n$. Then with this $\alpha$, we have
\begin{equation*}
D(p,X_{n+1}) \leq D(p,X_n) - \frac{K}{2}\left(1-\alpha^2 \right) \|Y_n - X_n\|^2.
\end{equation*}
Telescoping the above, we have
\begin{equation*}
D(p,X_{n+1}) \leq D(p,X_1) - \frac{K}{2}\left(1-\alpha^2 \right) \sum_{k=1}^{n} \|Y_k - X_k\|^2.
\end{equation*}
Rearranging the above, we have
\begin{equation*}
\frac{K}{2}\left(1-\alpha^2 \right) \sum_{k=1}^{n} \|Y_k - X_k\|^2 \leq D(p,X_1) - D(p,X_{n+1}) \leq D(p,X_1),
\end{equation*}
where the last inequality follows by positivity of Bregman divergence. Taking the limit as $n\to\infty$ on both sides of the above inequality, we have $\lim_{n\to\infty} \sum_{k=1}^{n} \|Y_k - X_k\|^2 < \infty$, which implies that $\lim_{n\to\infty} \|Y_n - X_n\|=0$.

Next by compactness of $\Xcal$, we have that $\{X_n\}_n$ has a convergent subsequence $\{X_{n_k}\}_k$ such that $\lim_{k\to\infty} X_{n_k} = \hat{x} \in \Xcal$. We show in the following that, in fact, $\hat{x} \in \Xcal^*$. First notice that 
\begin{equation*}
\lim_{k\to\infty} \|Y_{n_k} - \hat{x}\|  \leq  \lim_{k\to\infty} \left( \|Y_{n_k} - X_{n_k}\|  + \|X_{n_k} - \hat{x}\| \right)= 0.
\end{equation*}
Moreover, suppose that $\lim_n \gamma_n = \gamma$, and so $\lim_{k\to\infty} \gamma_{n_k} = \gamma$ as well. Then
\begin{equation*}
\hat{x} = \lim_{k\to\infty} Y_{n_k}  = \lim_{k\to\infty} P_{X_{n_k}}(-\gamma_{n_k} g(X_{n_k})) = P_{\hat{x}}(-\gamma g(\hat{x})),
\end{equation*}
where the last equality follow by
\begin{align*}
&\|P_{X_{n_k}}(-\gamma_{n_k} g(X_{n_k})) - P_{\hat{x}}(-\gamma g(\hat{x}))\| \\
&\qquad \leq \| P_{X_{n_k}}(-\gamma_{n_k} g(X_{n_k})) - P_{X_{n_k}}(-\gamma g(\hat{x})) \| + \| P_{X_{n_k}}(-\gamma g(\hat{x})) - P_{\hat{x}}(-\gamma g(\hat{x}))\|\\
& \qquad \overset{(a)}{\leq} \frac{1}{K}\|\gamma_{n_k}g(X_{n_k}) - \gamma g(\hat{x})\| + \frac{L_h}{K}\|X_{n_k} - \hat{x}\|\\
& \qquad \leq \frac{1}{K} \|\gamma_{n_k}g(X_{n_k}) - \gamma_{n_k}g(\hat{x})\| + \frac{1}{K} \|\gamma_{n_k}g(\hat{x})-\gamma g(\hat{x})\| + \frac{L_h}{K}\|X_{n_k} - \hat{x}\|\\
& \qquad \leq \frac{L_g}{K} |\gamma_{n_k}|\|X_{n_k} - \hat{x}\| + \frac{1}{K} |\gamma_{n_k}-\gamma | \|g(\hat{x})\| + \frac{L_h}{K}\|X_{n_k} - \hat{x}\|,
\end{align*}
since $\gamma_n<\infty$ for all $n$ and $\|g(\hat{x})\|<\infty$, taking the limit as $k\to\infty$ on both sides of the resulting inequality, we have that $\lim_{k\to\infty} \|P_{X_{n_k}}(-\gamma_{n_k} g(X_{n_k})) - P_{\hat{x}}(-\gamma g(\hat{x}))\| = 0$. In the above, step $(a)$ follows by \cite[Proposition B.4(a)]{Mertikopoulos2019optimistic} and Lemma \ref{lem:propB4ext}.
The above shows that $\hat{x}=P_{\hat{x}}(-\gamma g(\hat{x}))$. By \cite[(B.7)]{Mertikopoulos2019optimistic}, we have
\begin{equation*}
\langle  \grad h(\hat{x}), \hat{x} - x \rangle \leq \langle  \grad h(\hat{x}) - \gamma g(\hat{x}), \hat{x} - x  \rangle, \quad \forall x\in\Xcal.
\end{equation*}
That is, $\langle g(\hat{x}), x - \hat{x}\rangle\geq 0$ for all $x\in\Xcal$, which is (SVI). By definition of coherence, $\hat{x}$ must be a saddle point of (SP). 

So far, we have proved that $\lim_{k\to\infty} X_{n_k} = x^*\in\Xcal^*$. Next we want to show that $\lim_{n\to\infty} X_{n} = x^*\in\Xcal^*$. Similar to before, fix an $\alpha\in(0,1)$ such that $\gamma_n < \alpha K/L_g$. Then \eqref{eq:lemD1_1} in Lemma \ref{lem:D1} gives
\begin{equation*}
D(x^*,X_{n+1}) \leq D(x^*,X_n)  - \frac{K}{2}\left(1- \alpha^2\right) \|Y_n - X_n\|^2 - \gamma_n \langle g(Y_n), Y_n - x^*\rangle.
\end{equation*}
Telescoping the above, we have
\begin{align*}
D(x^*,X_{n+1}) & \leq D(x^*,X_{n_0})  - \frac{K}{2}\left(1- \alpha^2\right) \sum_{k=n_0}^n\|Y_k - X_k\|^2 - \alpha K/L_g \sum_{k=n_0}^n \langle g(Y_k), Y_k - x^*\rangle\\
& \leq D(x^*,X_{n_0}) - \alpha K/L_g \sum_{k=n_0}^n \langle g(Y_k), Y_k - x^*\rangle,
\end{align*}
where the choice of $n_0$ is as follows: let $\e := \bar{\e}/(2+L_h\diamX)$ with $\bar{\e}\in(0,\e_0)$ being arbitrary but fixed,
\begin{enumerate}
\item Choose $N_1$ sufficiently large such that $\|X_n - Y_n\| \leq \e$ for all $n\geq N_1$, such choice is possible since we have proved that $\|X_n - Y_n\| \to 0$ as $n\to\infty$;
\item Choose $N_2$ sufficiently large such that $D(X_n, Y_n) \leq \e$ for all $n\geq N_2$, such choice is possible since by $\|X_n - Y_n\| \to 0$ and the Bregman reciprocity condition, we have $D(X_n,Y_n)\to 0$ as $n\to \infty$;
\item Choose $n_0\geq \max\{N_1,N_2\}$ such that $D(x^*,X_{n_0})\leq \e$, such choice is possible since we have proved that $\lim_{k\to\infty} D(x^*, X_{n_k})=0$.
\end{enumerate}
With such choice of $n_0$, we will prove that $D(x^*, X_{n+1})\leq \e$ for all $n\geq n_0$.

First we show that for all $n \geq n_0$, 
\begin{equation}
D(x^*, X_n)\leq \e \Longrightarrow D(x^*, Y_n)\leq (2+L_h\diamX)\e = \bar{\e}. 
\label{eq:D_XnYn}
\end{equation}
To see this, in \cite[Lemma B.2]{Mertikopoulos2019optimistic}, letting $p=x^*, x'=Y_n, x=X_n$, we have
\begin{equation*}
\begin{split}
D(x^*, Y_n) &= D(x^*, X_n) + D(X_n, Y_n) + \langle \grad h(Y_n) - \grad h(X_n), X_n - x^* \rangle\\
&\leq D(x^*, X_n) + D(X_n, Y_n) + L_h \|Y_n - X_n\| \|X_n - x^*\|\\
&\leq \e + \e + L_h \diamX \e,
\end{split}
\end{equation*}
where the first inequality follows by Cauchy-Schwarz and the Lipschitz property of $\grad h$, and the second inequality follows by our choice of $n_0$.

Now starting with $n = n_0$, we have
\begin{equation*}
D(x^*, X_{n_0+1}) \leq D(x^*,X_{n_0}) - \alpha K/L_g \langle g(Y_{n_0}), Y_{n_0} - x^*\rangle.
\end{equation*}
Since $D(x^*,X_{n_0}) \leq \e$, by \eqref{eq:D_XnYn} we have $D(x^*, Y_{n_0})\leq \bar{\e}$. By our modified Condition 3 in the definition of coherence, we have $\langle g(Y_{n_0}), Y_{n_0} - x^*\rangle \geq 0$, which implies 
\begin{equation*}
D(x^*, X_{n_0+1}) \leq D(x^*,X_{n_0})  - \alpha K/L_g \langle g(Y_{n_0}), Y_{n_0} - x^*\rangle \leq D(x^*,X_{n_0}) \leq \e.
\end{equation*}
Using \eqref{eq:D_XnYn} again, the above implies $D(x^*, Y_{n_0+1})\leq \bar{\e}$ and hence $\langle g(Y_{n_0+1}), Y_{n_0+1} - x^*\rangle \geq 0$. Therefore,
\begin{equation*}
D(x^*, X_{n_0+2}) \leq D(x^*,X_{n_0})  - \alpha K/L_g \sum_{k=n_0}^{n_0+1}\langle g(Y_{k}), Y_{k} - x^*\rangle \leq D(x^*,X_{n_0}) \leq \e.
\end{equation*}
Keeping this procedure, we can show that for all $n\geq n_0$, we have $D(x^*, X_n)\leq \e$. Since $\e$ can be chosen to be arbitrarily close to zero (by choosing $\bar{\e}$ arbitrarily close to zero), we have proved that for all $\e>0$, there exists an $n_0(\e)$ such that $D(x^*, X_n)\leq \e$ for all $n\geq n_0(\e)$, hence $\lim_{n\to\infty} D(x^*, X_n)=0$. By the Bregman reciprocity condition, we have $\lim_{n\to\infty} X_n \to x^*$.
\end{proof}

\subsection{Proof of Theorem \ref{thm:SMD}}
\begin{proof}

\begin{enumerate}

\item[(a)] The same technique used for proving \textbf{(ii)} and \textbf{(iii)} in Theorem \ref{thm:SOMD} can also be used to prove the convergence of mirror descent (MD) algorithm; we omit the details here.

\item[(b)] Note that there is a typo in the proof of \cite[Theorem 3.1(b)]{Mertikopoulos2019optimistic} that can be quite confusing: in \cite[(C.14)]{Mertikopoulos2019optimistic}, the plus sign before the last innerproduct should be a minus sign. To see how \cite[(C.14)]{Mertikopoulos2019optimistic} (with the corrected sign) is obtained, we first recall that $h$ is proper, convex, and closed on $\Xcal$. Therefore, we have $\grad h^{-1} = \grad h^*$, where $h^*(x^*):=\sup_{x\in\Xcal} \langle x^*,x\rangle  - h(x)$ is the convex conjugate of $h$. By \cite[(B.5),(B.6b)]{Mertikopoulos2019optimistic}, MD \eqref{eq:SMD_algo} can be written as
\begin{equation*}
X_{n+1} = \grad h^*(\grad h(X_n) - \gamma_n \ghat_n).
\end{equation*}
It follows that 
\begin{equation*}
\grad h (X_{n+1}) = \grad h \left( \grad h^*(\grad h(X_n) - \gamma_n \ghat_n)\right) = \grad h(X_n) - \gamma_n \ghat_n,
\end{equation*}
hence
\begin{equation}
\grad h(X_{n+1}) - \grad h(X_n) = -\gamma_n \ghat_n.
\label{eq:gradh_gradg}
\end{equation}
Now applying \cite[Lemma B.2]{Mertikopoulos2019optimistic} with $p=p$ (a saddle points that satisfies Condition 2 of Definition \ref{def:coherence}), $x'=X_{n+1}$, and $x=X_n$, we have
\begin{equation*}
\begin{split}
D(p,X_{n+1}) &= D(p, X_{n}) + D(X_{n}, X_{n+1}) + \langle  \grad h(X_{n+1}) - \grad h(X_n),  X_{n} - p\rangle\\
&=D(p, X_{n}) + D(X_{n}, X_{n+1}) - \gamma_n \langle  \ghat_n,  X_{n} - p\rangle,
\end{split}
\end{equation*}
where the last equality follows by \eqref{eq:gradh_gradg}. Taking expectation on both sides,
\begin{equation*}
\begin{split}
\Ebb[D(p,X_{n+1})] &\overset{(a)}{=} \Ebb[D(p, X_{n}) ] + \Ebb[D(X_{n}, X_{n+1})] - \gamma_n \Ebb[ \langle g(X_n), X_{n} - p  \rangle]\\
&\overset{(b)}{=}\Ebb[D(p, X_{n}) ] + \Ebb[D(X_{n}, X_{n+1})] \\
&\geq \Ebb[D(p, X_{n}) ],
\end{split}
\end{equation*}
where step $(a)$ follows by $\Ebb[ \langle  \ghat_n,  X_{n} - p\rangle] = \Ebb [   \langle  \Ebb [ \ghat_n | \Fcal_n] ,  X_{n} - p\rangle  ] = \Ebb [   \langle  g(X_n) ,  X_{n} - p\rangle  ]$, since $X_{n}$ is $\Fcal_n$-measurable and $\ghat_n$ is an unbiased conditioned on $\Fcal_n$ by \eqref{eq:SMD_grad_cond}, and step $(b)$ follows by definition of null-coherence. This shows that the sequence $\{\Ebb[D(p,X_n)]\}_n$ is non-decreasing for the special saddle points that satisfy global (MVI).
\end{enumerate}
\end{proof}

\subsection{Proof of Theorem \ref{thm:SOMD}}

\begin{proof}

We first include a result from \cite[Proposition B.4(b)]{Mertikopoulos2019optimistic}, which will be used frequently in the proof of the theorem.
\begin{lemma}[{\cite[Proposition B.4(b)]{Mertikopoulos2019optimistic}}]
Let $K$ be defined in \eqref{eq:def_K} and let the prox-mapping $P_x$ be defined in \eqref{eq:def_Proj}. Fix some $x'\in\Xcal$, $x\in\text{dom}\grad h$. Letting $x_1^+=P_x(y_1)$ and $x_2^+=P_x(y_2)$, we have
\begin{equation*}
    D(x',x_2^+) \leq D(x',x) + \langle y_2, x_1^+ - x'\rangle + \frac{1}{2K}\|y_1 - y_2\|_{*}^2 - \frac{K}{2}\|x_1^+ - x\|^2.
\end{equation*}
\label{lem:B22b}
\end{lemma}

We are now ready to prove the theorem. The proof contains three steps: 
\begin{enumerate}
\item[\textbf{(i)}] Show that 
\begin{equation}
\Pbb\left(   \lim_{n\to\infty} \|X_n - Y_n\| = 0  \right) = 1.
\label{eq:XnYn_converge}
\end{equation}

\item[\textbf{(ii)}] Let $\{Y_{n_k}\}_{k}$ denote a subsequence of $\{Y_n\}_n$. Show that 
\begin{equation}
\Pbb\left(\exists \{Y_{n_k}\}_{k}, \text{ s.t. } \lim_{k\to\infty}\inf_{x^*\in\Xcal^*} \|Y_{n_k} - x^*\| = 0 \right) = 1.
\label{eq:Ynk_converge}
\end{equation}

\item[\textbf{(iii)}] Show that \eqref{eq:SOMDresult} holds.
\end{enumerate} 
Let $\{X_{n_k}\}_{k}$ denote a subsequence of $\{X_n\}_n$ and we notice that \textbf{(i)} and $\textbf{(ii)}$ imply that
\begin{equation}
\Pbb\left(\exists \{X_{n_k}\}_{k}, \text{ s.t. } \lim_{k\to\infty}\inf_{x^*\in\Xcal^*} \| X_{n_k} - x^*\| = 0 \right) = 1.
\label{eq:Xnk_converge}
\end{equation}
To see this, denote the events considered in \eqref{eq:XnYn_converge}, \eqref{eq:Ynk_converge}, and \eqref{eq:Xnk_converge} by $E$, $E_Y$, and $E_X$, respectively. For any $\omega \in E_Y \cap E$, 
\begin{align*}
\inf_{x^*\in\Xcal^*} \|X_{n_k}(\omega) - x^*\| &\leq \inf_{x^*\in\Xcal^*} \|X_{n_k}(\omega) Y_{n_k}(\omega) + Y_{n_k}(\omega)- x^*\|\\
&\leq \|X_{n_k}(\omega) - Y_{n_k}(\omega)\|  + \inf_{x^*\in\Xcal^*} \|Y_{n_k}(\omega) - x^*\|.
\end{align*}
Letting $k\to\infty$ on both sides of the resulting inequality, we have $\lim_{k\to\infty} \inf_{x^*\in\Xcal^*} \| X_{n_k}(\omega) - x^*\|= 0$, thus $\omega\in E_X$. To recap, we have shown $\omega\in E_Y \cap E$ implies $\omega\in E_X$, which implies $E_Y \cap E \subset E_X$, and so $\Pbb(E_X)\geq \Pbb(E_Y \cap E)=1$; \eqref{eq:Xnk_converge} is proved. We will see later that \eqref{eq:Xnk_converge} will be used to prove \textbf{(iii)}. 

We now show \textbf{(i)}. 
Let $p$ be one of the special saddle points that satisfy global (MVI). Define
\begin{equation*}
\begin{split}
U_{n+1}^+&:= \rhat_n - g(Y_n),\quad \text{and}\quad \xi_{n+1}^+:= -\langle U_{n+1}^+, Y_n - p \rangle.\\
U_{n+1}&:= \ghat_n - g(X_n),\quad \text{and}\quad \xi_{n+1}:= -\langle U_{n+1}, X_n - p \rangle.
\end{split}
\end{equation*}

Let $x'=p$ (a saddle point that satisfies Condition 2 in Definition \ref{def:coherence}), $x=X_n$, $y_1=-\gamma_n \ghat_n$, $x_1^+=Y_n$, $y_2=-\gamma_n \rhat_n$, and $x_2^+=X_{n+1}$ in Lemma~\ref{lem:B22b},
then
\begin{align*}
D(p,X_{n+1}) &\leq D(p,X_n) - \gamma_n\langle \rhat_n, Y_n - p\rangle + \frac{\gamma_n^2}{2K}\|\rhat_n - \ghat_n\|_*^2 - \frac{K}{2} \|Y_n - X_n\|^2\\
&\leq D(p,X_n) - \gamma_n\langle g(Y_n), Y_n - p\rangle - \gamma_n\langle U_{n+1}^+, Y_n - p\rangle\\
&\qquad +\frac{\gamma_n^2}{K}\|\rhat_n\|_*^2 + \frac{\gamma_n^2}{K} \|\ghat_n\|_*^2 - \frac{K}{2} \|Y_n - X_n\|^2\\
&\leq D(p,X_n) + \gamma_n\xi_{n+1}^+ +\frac{\gamma_n^2}{K}\|\rhat_n\|_*^2 + \frac{\gamma_n^2}{K} \|\ghat_n\|_*^2 - \frac{K}{2} \|Y_n - X_n\|^2,
\end{align*}
where the last inequality follows by (MVI).
Telescoping the above, we have 
\begin{align}
\frac{K}{2} \sum_{k=1}^{n} \|Y_k - X_k\|^2 &\leq D(p,X_1) - D(p,X_{n+1}) + \sum_{k=1}^n \gamma_k \xi_{k+1}^+ + \frac{1}{K} \sum_{k=1}^n \gamma_k^2 \|\ghat_k\|_*^2 + \frac{1}{K} \sum_{k=1}^n \gamma_k^2 \|\rhat_k\|_*^2\nonumber\\
&\leq D(p,X_1)  + \sum_{k=1}^n \gamma_k \xi_{k+1}^+ + \frac{1}{K} \sum_{k=1}^n \gamma_k^2 \|\ghat_k\|_*^2 + \frac{1}{K} \sum_{k=1}^n \gamma_k^2 \|\rhat_k\|_*^2.
\label{eq:three_terms}
\end{align}
We will show that each of the last three terms on the RHS converges to some random variable almost surely as $n\to\infty$ and that random variable is finite almost surely. Then we can conclude 
\begin{equation*}
\sum_{k=1}^{\infty}  \|Y_k - X_k\|^2  < \infty \quad \text{a.s.,}
\end{equation*}
which implies that
\begin{equation}
\lim_{k\to\infty} \|Y_k - X_k\| = 0 \quad \text{a.s.}.
\label{eq:stoOMD_XY_converge}
\end{equation}
First consider the second term in \eqref{eq:three_terms}. Define $M_n:=\sum_{k=1}^{n-1} \gamma_k \xi_{k+1}^+$ and notice that for $m>n$,
\begin{align*}
\Ebb[M_{m} | \Fcal_{n,n}] = \sum_{k=1}^{n-1} \gamma_{k}\Ebb[\xi_{k+1}^+ | \Fcal_{n,n}]  + \sum_{k=n}^{m-1} \gamma_{k}\Ebb[\xi_{k+1}^+ | \Fcal_{n,n}] = M_n + 0,
\end{align*}
where the last equality follows by the fact that $\xi_k^+$ is $\Fcal_{n,n}$ measurable for $k\leq n$ and $\Ebb[\xi_k^+ | \Fcal_{n,n}] = \Ebb[\Ebb[\xi_k^+|\Fcal_{k-1,k-1}] | \Fcal_{n,n}] =0$ for $k\geq n+1$. Hence, $\{M_n,\Fcal_{n,n}\}_n$ is a martingale. Note also that
\begin{align*}
\Ebb [M_n^2] &= \sum_{k=1}^{n-1} \gamma_k^2 \Ebb[(\xi_{k+1}^+)^2] + \sum_{k=1}^{n-1} \sum_{s\neq k}\gamma_k^2 \gamma_s^2\Ebb[\xi_{k+1}^+ \xi_{s+1}^+] \overset{(a)}{=} \sum_{k=1}^{n-1} \gamma_k^2 \Ebb[(\xi_{k+1}^+)^2]\\
&\leq \sum_{k=1}^{n-1} \gamma_k^2 \Ebb[\|U_{k+1}^+\|_*^2 \|Y_k - p\|^2]\leq \diamX^2 \sum_{k=1}^n \gamma_k \Ebb[\Ebb[\|U_{k+1}^+\|_*^2 | \Fcal_{k,k}]]\\ &\leq \diamX^2\sigma^2 \sum_{k=1}^n \gamma_k^2,
\end{align*}
where step $(a)$ follows by $\Ebb[\xi_{k+1}^+ \xi_{s+1}^+] = \Ebb[\Ebb[\xi_{k+1}^+ | \Fcal_{k,k}] \xi_{s+1}^+]=0$ for $s<k$. (If $s>k$, then first condition on $\Fcal_{s,s}$.)
It follows that $\sup_n \Ebb[M_n^2] \leq \sup_n \diamX^2\sigma^2 \sum_{k=1}^n \gamma_k^2 < \infty$. Then we also have $\sup_n \Ebb[|M_n|]<\infty$, since $\Ebb[|M_n|] \leq (\Ebb[M_n^2])^{1/2}$ by H{\" o}lder's inequality. Then by the martingale convergence theorem, we have that $\lim_{n\to\infty}M_n=M$ exists almost surely and $\Pbb(M<\infty) = 1$. 
Next consider the third term in \eqref{eq:three_terms}. Define $S_n:=\sum_{k=1}^{n-1} \gamma_k^2 \|\ghat_k\|_*^2$ and for $m>n$,
\begin{equation}
\Ebb[S_m |\Fcal_{n,n-1}] = \sum_{k=1}^{n-1} \gamma_k^2 \Ebb[ \| \ghat_{k}\|_*^2 | \Fcal_{n,n-1} ] +  \sum_{k=n}^{m-1} \gamma_k^2 \Ebb[ \|\ghat_{k}\|_*^2 | \Fcal_{n,n-1} ] \leq S_n + G^2 \sum_{k=n}^{m-1} \gamma_k^2,
\label{eq:Sm}
\end{equation}
where the last inequality follows by $\ghat_k$ being $\Fcal_{n,n-1}$-measurable for $k\leq n-1$ and the assumption in \eqref{eq:grad_assumption}.
Now let $R_n:=S_n + G^2 \sum_{k=n}^\infty \gamma_k^2$ and notice that
\begin{equation*}
\Ebb[R_m | \Fcal_{n,n-1}] = \Ebb[S_m | \Fcal_{n,n-1}] + G^2\sum_{k=m}^{\infty} \gamma_k^2 \leq S_n + G^2 \sum_{k=n}^{m-1} \gamma_k^2 + G^2\sum_{k=m}^{\infty} \gamma_k^2 = S_n + G^2 \sum_{k=n}^{\infty} = R_n,
\end{equation*}
which implies that $\{R_n, \Fcal_{n,n-1}\}_{n}$ is a super-martingale. Note that
\begin{equation*}
\Ebb[|R_n|] \leq  \sum_{k=1}^{n-1} \gamma_k^2 \Ebb[\Ebb[\|\ghat_k\|_*^2 | \Fcal_{k,k-1}]] + G^2 \sum_{k=n}^\infty \gamma_k^2 \leq G^2 \sum_{k=1}^\infty \gamma_k^2 < \infty.
\end{equation*}
Hence, by the super-martingale convergence theorem, we have that $\lim_{n\to\infty} R_n =R$ exists almost surely and $\Pbb(R<\infty)=1$. Similarly, we can show that the fourth term in \eqref{eq:three_terms} also converges to a random variable almost surely and that random variable is finite with probability 1. It follows that the RHS of \eqref{eq:three_terms} is finite almost surely when $n\to\infty$, which confirms our assertion stated in \eqref{eq:stoOMD_XY_converge}.

Next we show \textbf{(ii)}. 
Note that by definition of limit, the event being measured in \eqref{eq:Ynk_converge} is equivalent to
\begin{equation*}
    \Big\{\exists \{Y_{n_k}\}_k, \text{ s.t. } \forall \e>0, \exists k_0, \text{ s.t. } \inf_{x^*\in \Xcal^*} \|Y_{n_k} - x^*\|\leq \e, \forall k\geq k_0 \Big\}.
\end{equation*}
Therefore, to show \eqref{eq:Ynk_converge} it is equivalent to show that for any $\e>0$, $\{Y_n\}_n$ enters the $\e$-neighborhood of $\Xcal^*$, which is defined as $B(\Xcal^*,\e):=\{x\in\Xcal:\inf_{x^*\in\Xcal^*} \|x-x^*\|<\e\}$, infinitely often.
Assume by contradiction that $\{Y_n\}_n$ only enters $B(\Xcal^*,\e)$ finitely often. Then there exists an $n_0$ such that $\inf_{x^*\in\Xcal^*}\|Y_n - x^*\|\geq \e$ for all $n\geq n_0$. 
Note that $\Xcal\setminus B(\Xcal^*,\e)$ is compact and $g$ is continuous. Moreover, by definition of strict coherence, the point $p$ defined in Condition 2 of Definition \ref{def:coherence} satisfies $\langle g(x),x-p\rangle = 0$ only if $x\in\Xcal^*$. This implies that there exists some $a>0$ such that 
\begin{equation*}
\langle g(x), x - p \rangle \geq a, \quad \forall x \in \Xcal\setminus B(\Xcal^*,\e).
\end{equation*}
This implies that, 
\begin{equation}
\langle g(Y_n), Y_n - p \rangle \geq a, \quad \forall n \geq n_0.
\label{eq:compare_a}
\end{equation}
Using Lemma~\ref{lem:B22b} with $x'=p$ (a saddle point that satisfies Condition 2 of Definition \ref{def:coherence}), $x=X_n$, $y_1=-\gamma_n\ghat_n$, $y_2=-\gamma_n\rhat_n$, $x_1^+=Y_n$, and $x_2^+=X_{n+1}$, we have
\begin{align}
D(p, X_{n+1}) &\leq D(p, X_n) - \gamma_n \langle \rhat_n, Y_n - p \rangle + \frac{\gamma_n^2}{2K} \|\rhat_n - \ghat_n\|_*^2\nonumber\\
&= D(p, X_n) - \gamma_n \langle g(Y_n), Y_n - p \rangle - \gamma_n \langle U_{n+1}^+, Y_n - p \rangle + \frac{\gamma_n^2}{2K} \|\rhat_n - \ghat_n\|_*^2\nonumber\\
&= D(p, X_n) - \gamma_n \langle g(Y_n), Y_n - p \rangle + \gamma_n\xi_{n+1}^+ + \frac{\gamma_n^2}{2K} \|\rhat_n - \ghat_n\|_*^2,
\label{eq:basic}
\end{align}
Telescoping \eqref{eq:basic} and using \eqref{eq:compare_a}, we have
\begin{align}
D(p,X_{n+1}) &\leq D(p, X_{n_0}) - a \sum_{k=n_0}^n \gamma_k  + \sum_{k=n_0}^n \gamma_k \xi_{k+1}^+ + \sum_{k=n_0}^n \frac{\gamma_k^2}{2K} \|\rhat_k - \ghat_k\|_*^2\nonumber\\
&\leq D(p, X_{n_0}) - a \sum_{k=n_0}^n \gamma_k  + \sum_{k=n_0}^n \gamma_k \xi_{k+1}^+ + \frac{1}{K}\sum_{k=n_0}^n \gamma_k^2 \|\rhat_k\|_*^2  + \frac{1}{K}\sum_{k=n_0}^n \gamma_k^2\| \ghat_k\|_*^2.
\label{eq:contradiction}
\end{align}
Note that the last three terms in \eqref{eq:contradiction} are similar to the last three term in \eqref{eq:three_terms} (only the lower limit of the summations is different), and so we can use the same method to show that they stay finite with probability 1 as $n\to\infty$. Note also that the first term in \eqref{eq:contradiction} is finite and does not change with $n$, moreover, the second term in \eqref{eq:contradiction}  goes to $-\infty$ as $n\to\infty$. Hence, \eqref{eq:contradiction} implies that $\lim_{n\to\infty} D(p,X_{n+1})\leq -\infty$, which contradicts with the fact that Bregman divergence is non-negative. Therefore, our initial assumption that $Y_n$ enters $B(\Xcal^*,\e)$ only finitely often is incorrect, and so \eqref{eq:Ynk_converge} is proved.

Finally we show \textbf{(iii)}. Recall that we have obtained the result in \eqref{eq:Xnk_converge}. By the Bregman reciprocity condition, we have $\lim_{k\to\infty}\inf_{x^*\in\Xcal^*} D(x^*, X_{n_k})=0$ with probability 1 as well.
Now replacing $p$ with any $x^*\in\Xcal^*$ in \eqref{eq:basic}, 
\begin{align}
D(x^*, X_{n+1}) &\leq  D(x^*, X_n) - \gamma_n \langle g(Y_n), Y_n - x^* \rangle + \gamma_n\xi_{n+1}^+ + \frac{\gamma_n^2}{2K} \|\rhat_n - \ghat_n\|_*^2\nonumber\\
& \leq D(x^*, X_n) - \gamma_n \langle g(Y_n), Y_n - x^* \rangle + \gamma_n\xi_{n+1}^+ + \frac{\gamma_n^2}{K} \|\rhat_n\|_*^2 + \frac{\gamma_n^2}{K}  \|\ghat_n\|_*^2,\label{eq:basic2}
\end{align}
Following similar procedure as in the proof of  \cite[Theorem 5.2]{Zhou2017stochasticMD}, we have that for any $\delta\in(0,1)$ and $\e>0$, if the step-size sequence $\{\gamma_n\}_n$ satisfies \eqref{eq:gamma_cond}, then for any fixed $n_0\in\mathbb{N}$,
\begin{equation*}
\begin{split}
\Pbb\left(  \sup_{n\geq n_0} \sum_{k=n_0}^n \gamma_k\xi_{k+1}^+ \leq \e \right) &\geq 1 - \frac{\delta}{3}\\
\Pbb\left(  \sup_{n\geq n_0} \frac{1}{K}\sum_{k=n_0}^n \gamma_k^2 \|\rhat_k\|_*^2 \leq \e \right) &\geq 1 - \frac{\delta}{3}\\
\Pbb\left(  \sup_{n\geq n_0} \frac{1}{K}\sum_{k=n_0}^n \gamma_k^2 \|\ghat_k\|_*^2 \leq \e \right) &\geq 1 - \frac{\delta}{3}.
\end{split}
\end{equation*}
Let the three sets measured above be denoted by $A_1$, $A_2$, and $A_3$, respectively, and let $B_{\delta}:=\cap_{i=1}^3 A_i$. Then
\begin{equation*}
\Pbb\left( B_{\delta} \right) = 1 - \Pbb\left(  \left( \cap_{i=1}^3 A_i \right)^c \right) = 1 - \Pbb\left(  \cup_{i=1}^3 A_i^c  \right)\geq 1- \sum_{i=1}^3 \Pbb(A_i^c) \geq 1 - \delta.
\end{equation*}
Let $\e = \bar{\e}/(2+L_h \diamX)$ for some arbitrary but fixed $\bar{\e}\in(0,\e_0/4)$. Recall that we have defined a set $E$ to be the event considered in \eqref{eq:XnYn_converge}.
Now for an arbitrary but fixed $\omega\in E \cap B_\delta$, choose $n_0$ following the three steps listed above \eqref{eq:D_XnYn}, with Step 3 being replaced by $\inf_{x'\in\Xcal^*} D(x',X_{n_0}(\omega))\leq \e$. Let $x^*\in\Xcal^*$ be the point that achieves the infimum. Then we have $D(x^*, X_{n_0}(\omega))\leq \e$ and $D(x^*, Y_{n_0}(\omega))\leq \bar{\e}$. 
Telescoping \eqref{eq:basic2}, we have 
\begin{align}
D(x^*, X_{n+1}(\omega)) &\leq D(x^*, X_{n_0}(\omega)) - \sum_{k=n_0}^n \gamma_k \langle g(Y_k(\omega)), Y_k(\omega) - x^* \rangle\nonumber\\
& \qquad + \sum_{k=n_0}^n \gamma_k\xi_{k+1}^+(\omega) + \frac{1}{K} \sum_{k=n_0}^n \gamma_k^2 \|\rhat_k(\omega)\|_*^2 + \frac{1}{K}\sum_{k=n_0}^n \gamma_k^2  \|\ghat_k(\omega)\|_*^2\nonumber\\
&\leq 4\e  -  \sum_{k=n_0}^n \gamma_k \langle g(Y_k(\omega)), Y_k(\omega)-x^*\rangle.
\label{eq:for_induction}
\end{align}
From \eqref{eq:for_induction}, we prove by induction that for all $n\geq n_0$, $D(x^*, X_{n+1}(\omega))\leq 4\e$. First, for $n=n_0$, we have
\begin{equation*}
D(x^*, X_{n+1}(\omega)) \leq 4\e  -  \gamma_{n_0} \langle g(Y_{n_0}(\omega)), Y_{n_0}(\omega)-x^*\rangle.
\end{equation*}
As we have explained above that $D(x^*, Y_{n_0}(\omega))\leq \bar{\e}<\e_0$, then by the local (MVI) assumption, we have $\langle g(Y_{n_0}(\omega)), Y_{n_0}(\omega)-x^*\rangle\geq 0$, and so $D(x^*, X_{n_0+1}(\omega)) \leq 4\e $. Now assume that $D(x^*, X_{n}(\omega))\leq 4\e$ for $n=n_0,n_0+1,\ldots,N$ for some $N>n_0$, which by \eqref{eq:D_XnYn} implies $D(x^*, Y_{n}(\omega))\leq 4 \bar{\e} < \e_0$ for $n=n_0,n_0+1,\ldots,N$ and thus $ \langle g(Y_n(\omega)), Y_n(\omega)-x^*\rangle\geq 0$ for $n=n_0,n_0+1,\ldots,N$, which implies $D(x^*, X_{N+1}(\omega)) \leq 4\e$; we have completed the inductive proof. 

To recap, we have proved that for any fixed $\delta\in(0,1)$ and $\e>0$, choose $\{\gamma_n\}_n$ that satisfies \eqref{eq:gamma_cond}, then for all $\omega \in E\cap B_{\delta}$, there exists an $n_0(\e,\delta,\omega)<\infty$, such that $D(x^*, X_n(\omega))\leq 4\e$ for all $n\geq n_0(\e,\delta,\omega)$. Then the theorem statement follows by noticing that $\Pbb(E \cap B_{\delta}) \geq 1- \delta$.
\end{proof}

\end{document}